\documentclass{article}
\usepackage[utf8]{inputenc}
\usepackage{amssymb,amsmath,amsthm}
\usepackage{hyperref}

\newtheorem{theorem}{Theorem}
\newtheorem{claim}[theorem]{Claim}
\newtheorem{lemma}[theorem]{Lemma}

\newtheorem{conjecture}{Conjecture}

\newcommand{\floor}[1]{\left\lfloor{#1}\right\rfloor}
\newcommand{\ceil}[1]{\left\lceil{#1}\right\rceil}

\usepackage{authblk}
\DeclareMathOperator{\ex}{ex}

\title{Extremal results for graphs avoiding a rainbow subgraph}

\author[1]{Peter Frankl}
\author[1]{Ervin Győri}
\author[1,2]{Zhen He}
\author[1,2]{Zequn Lv}
\author[1,4]{Nika Salia}
\author[1]{Casey Tompkins}
\author[1,5]{Kitti Varga}
\author[1,3]{Xiutao Zhu}
\date{}

\affil[1]{Alfr\'ed R\'enyi Institute of Mathematics, Hungarian Academy of Sciences. }
\affil[2]{Department of Mathematical Sciences, Tsinghua University.}
\affil[3]{Department of Mathematics, Nanjing University. }
\affil[4]{Extremal Combinatorics and Probability Group, Institute for Basic Science, Daejeon, South Korea.}
\affil[5]{Department of Computer Science and Information Theory, Budapest University of Technology and Economics.}
\begin{document}

\maketitle

\begin{abstract}
We say that $k$ graphs $G_1,G_2,\dots,G_k$ on a common vertex set of size $n$ contain a rainbow copy of a graph $H$ if their union contains a copy of $H$ with each edge belonging to a distinct $G_i$. 
We provide a counterexample to a conjecture of Frankl on the maximum product of the sizes of the edge sets of three graphs avoiding a rainbow triangle. 
We propose an alternative conjecture, which we prove under the additional assumption that the union of the three graphs is complete. 
Furthermore, we determine the maximum product of the sizes of the edge sets of three graphs or four graphs avoiding a rainbow path of length three.

\end{abstract}

The classical theorem of Mantel~\cite{mantel} asserts that the maximum number of edges in an $n$-vertex graph containing no triangle is $\floor{n^2/4}$. 
This result was generalized by Tur\'an~\cite{turan}, who showed that the maximum number of edges in an $n$-vertex graph with no complete graph $K_r$ as a subgraph is obtained by taking a complete $(r-1)$-partite graph with parts of size $\floor{n/(r-1)}$ or $\ceil{n/(r-1)}$.  

Many natural generalizations of these theorems have been considered.
For a graph $G$, let $E(G)$ denote the edge set of $G$ and let $e(G) = |E(G)|$.
Of particular importance to the present work is an extremal problem due to Keevash, Saks, Sudakov and Verstra\"ete~\cite{keevash}. 
They considered (among other problems) the maximum of $e(G_1)+e(G_2)+\cdots+e(G_k)$ across $k$ graphs on a common vertex set of size $n$ with the property that there is no $K_r$ with each of its edges coming from a distinct $E(G_i)$.   
Such a $K_r$ is referred to as a \emph{rainbow}~$K_r$. 
For this problem, there are two natural constructions.
 On the one hand, if $k\ge \frac{r^2-1}{2}$, then we take $k$ identical copies of the Tur\'an graph. 
On the other hand, when $\binom{r}{2}\le k< \frac{r^2-1}{2}$, it is better to take $\binom{r}{2}-1$ copies of the complete graph and let the remaining graphs have empty edge sets (of course for $k<\binom{r}{2}$, it is optimal to take all $k$ graphs to be complete).
Keevash, Saks, Sudakov and Verstra\"ete~\cite{keevash} also proved that for $3$-chromatic color-critical graphs and sufficiently large $n$, an analogous result holds in which either a construction consisting of $k$ copies of a Tur\'an graph, or a construction consisting of complete graphs and graphs with no edges is optimal. 
Recently Chakraborti, Kim, Lee, Liu and Seo~\cite{deb} showed that the same holds for $4$-chromatic color-critical graphs and almost all color-critical graphs of chromatic number at least $5$, partially verifying a conjecture from~\cite{keevash}. 
In the context of extremal set theory, rainbow extremal problems have also been considered earlier, for example by Hilton~\cite{hilton}.

After one has bounds on $e(G_1)+e(G_2)+\cdots+e(G_k)$, it is natural to consider maximizing other objective functions over $e(G_1)$, $e(G_2)$,\dots, $e(G_k)$.   The problem of maximizing $\min\big(e(G_1), e(G_2), e(G_3)\big)$ while avoiding a rainbow triangle was considered by Aharoni, DeVos, de la Maza, Montejano and \v{S}\'amalin~\cite{ah}, answering a question of Diwan and Mubayi~\cite{diwan}. Let $P_r$ denote the path with $r$ vertices. The problem of maximizing $\min\big(e(G_1), e(G_2),\dots e(G_k)\big)$ while avoiding a rainbow $P_4$ was considered by Babi\'nski and Grzesik~\cite{Ba}.
For the problem of maximizing the product $e(G_1)e(G_2)e(G_3)$ while avoiding a rainbow triangle, Frankl~\cite{frankl} gave the following conjecture.

\begin{conjecture}[Frankl] \label{frankl}
Let $G_1,G_2,G_3$ be graphs on a common vertex set of size $n$ with no rainbow triangle. 
Then
\[
e(G_1)e(G_2)e(G_3) \le \floor{\frac{n^2}{4}}^3.
\]
\end{conjecture}
Taking $G_1,G_2,G_3$ to be $3$ copies of the complete bipartite graph with almost equal parts attains this bound. 
Frankl proved that under the additional assumption $E(G_1) \subseteq E(G_2)$ and $E(G_1) \subseteq E(G_3)$, Conjecture~\ref{frankl} holds.  
We show that Frankl's conjecture does not hold in the general case.  

Let $\gamma$ be the maximum of
\begin{equation} \label{opto}
\frac{x^2}{2}\left(\frac{x^2}{2}+\frac{(1-x)^2}{2}\right)\left(x(1-x)+\frac{(1-x)^2}{2}\right)
\end{equation}
on $[0,1]$, and assume $\gamma$ is attained at $x=x_0$. Note that $\frac{1}{52}<\gamma<\frac{1}{51}$ (and $x_0 \approx .729$).  We have the following.

\begin{theorem} \label{main}
There exist graphs $G_1,G_2,G_3$ on a common vertex set of size $n$ and with no rainbow triangle such that
\[
e(G_1)e(G_2)e(G_3) \ge \gamma n^6 \big( 1 - o(1) \big).
\]
\end{theorem}
 
\begin{proof}
Let $[n]=X\cup Y$ be a partition of $[n]$ with $X$ having size approximately~$x_0 n$ and $Y$ having size approximately $(1-x_0)n$.  Let $G_1$ consist of a complete graph on $X$, and $G_2$ consist of the union of a complete graph on $X$ and a complete graph on $Y$ and let $G_3$ consist of a complete graph on $Y$ as well as all edges between $X$ and $Y$.
Observe that the product $e(G_1)e(G_2)e(G_3)$ is asymptotically $\gamma n^6$. 
\end{proof}
Moreover, we believe that the expression in Theorem~\ref{main} is asymptotically best possible.  
\begin{conjecture} \label{mainconj}
For three graphs $G_1,G_2,G_3$ on a common vertex set of size $n$ with no rainbow triangle, 
\[
e(G_1)e(G_2)e(G_3) \le \gamma n^6 \big( 1+o(1) \big).
\]
\end{conjecture}

Conjecture~\ref{mainconj} is proved under the additional assumption that every pair of vertices forms an edge in at least one of the $E(G_i)$. 
For convenience, we say that a pair of vertices (or an edge) in the $n$-vertex ground set is colored if it belongs to at least one of  the sets $E(G_1)$, $E(G_2)$, $E(G_3)$. 
An edge is $t$-colored if it belongs to exactly $t$ of the sets $E(G_1)$, $E(G_2)$, $E(G_3)$. 

\begin{theorem} \label{sub}
Let $G_1,G_2,G_3$ be graphs on a common vertex set of size $n$ with no rainbow triangle. If $n$ is sufficiently large and every pair of vertices on the ground set is colored, then the construction described in Theorem~\ref{main} maximizes $e(G_1)e(G_2)e(G_3)$.
\end{theorem}
	
\begin{proof}
Let $n$ be sufficiently large, and let $G_1,G_2,G_3$ be graphs on a common vertex set of size $n$ with no rainbow triangle and assume that every edge is colored and $e(G_1)e(G_2)e(G_3)$ is maximal. 
For a vertex $v$ and $i \in [3]$, let $N_i(v)$ be the set of neighbors of $v$ in $G_i$ which are not neighbors of $v$ in $G_l$ for all $l\neq i$. For any $\{i,j,l\}=\{1,2,3\}$, we denote the set of neighbors of the vertex $v$ in $G_i$ and $G_j$ but not in $G_l$ by $N_{ij}(v)$.

 Assume $e(G_1)e(G_2)e(G_3) \ge \gamma n^6 \big( 1 - o(1) \big)$. Then we have 
 \[ e(G_1)+e(G_2)+e(G_3)\ge 3\sqrt[3]{e(G_1)e(G_2)e(G_3)}>0.8n^2+\frac{n}{2} \]
 for any sufficiently large $n$. 
 Since three-colored edges cannot be adjacent to any edge with at least two colors, the number of three-colored edges is at most $n/2$. 
 Hence the number of two-colored edges is at least $0.3 n^2$.
 	    
\begin{claim}\label{Claim:main}
The graph containing all edges with at least two colors is the union of vertex disjoint cliques such that every edge of each clique has the same coloring.
\end{claim} 
\begin{proof}
Let $e=uv$ be a two-colored edge, and assume that $e$ has colors $1$ and $2$. 
If $w\in N_3(v)$, then $uw$ is one-colored with color $3$.  That is  $N_3(v)=N_3(u)$ and 
\begin{equation}\label{equation:1}
   N_{13}(u)=N_{23}(u)=\emptyset.  
\end{equation}

Let $w\in N_{12}(v)$. Then $uw$ is not colored with color $3$.
Suppose by the way of contradiction that $uw$ is a one-colored edge. Without loss of generality, we may assume that $uw$ is of color $1$.
Then by the maximality of the coloring, there is a vertex $w'$ such that $w'\notin \{v,u,w\}$ and edges  $uw'$ and $ww'$ are colored with colors $1$ and $3$ in any order.
This is a contradiction since for any coloring of the edge $w'v$, the triangle $w'vu$ or the triangle $w'vw$ is a rainbow triangle.
Hence $w\in N_{1,2}(u)$, thus 
\begin{equation}\label{equation:2}
   N_{12}(u)=N_{12}(v).
\end{equation} 
Then the claim follows from~\eqref{equation:1} and~\eqref{equation:2}.
\end{proof}
   
Let $A$ be a clique of maximum size in the graph consisting of edges of at least two colors.
Let $a$ be the size of $A$. 
Then $\frac{(a-1)n}{2} \ge 0.3 n^2$ since the maximum degree in the graph of the two-colored edges is $a-1$ and we observed earlier that there are at least~$0.3n^2$ two-colored edges.
Hence we have $a\ge 0.6n$. 
Since $a>n/2$ and there are at least $0.3n^2$ two-colored edges, it follows $\binom{a}{2}+\binom{n-a}{2}\ge 0.3 n^2$. 
Thus we have~$a\ge 0.723n$. 

Without loss of generality, we can assume that the edges of $A$ are colored with colors $1$ and $2$. Then we have $e(A)\ge 0.26n^2$ and it follows $e(G_3)\le 0.24n^2$.
From the maximality of the product $e(G_1)e(G_2)e(G_3)$ and Claim~\ref{Claim:main}, we have that all one-colored edges are of color~$3$.
Indeed, otherwise we could change the color of all the one-colored edges to color $3$ and since $e(G_3)<e(A)$, this would increase the product.  
Moreover, for any maximal clique $B$ in the graph different from $A$ and consisting of edges of at least two colors, one of those colors must be color~$3$ (for otherwise, changing one of the colors in $B$ to color~$3$ would increase the product).  

By maximality, it is easy to observe that there is at most one clique colored with colors $i$ and $j$ for all $1\leq i<j\leq 3$. Let $a$ be the size of clique colored with colors $1$ and $2$, let $b$ be the size of clique colored with colors $1$ and $3$, and let $c$ be the size of clique colored with colors $2$ and $3$.
Let $d$ be the number of three-colored edges. 
Then $G_1$ consists of a complete graph of size $a$,  a complete graph of size $b$ and a matching of size $d$, the graph $G_2$ consists of complete graphs of size $a$ and $c$ and a matching of size $d$, and $G_3$ consists of all edges not in $A$. Without loss of generality we may assume $c\ge b$. We may also assume that we do not have the case $b=c=0$ for otherwise the product is only $O(n^5)$.  Then the following holds:
\[
\left(\binom{a}{2}+\binom{b}{2}+d\right)\left(\binom{a}{2} +\binom{c}{2}+d \right)\le \left(\binom{a}{2} +\binom{b}{2} \right) \left(\binom{a}{2} +\binom{c+2d}{2} \right).
\]
Therefore $d=0$, otherwise by changing the coloring, we could increase the product. Thus $a+b+c=n$ and since $a\ge 0.72n$, we have
\[
\left(\binom{a}{2}+\binom{b}{2}\right) \left(\binom{a}{2}+\binom{c}{2}\right)\le \left(\binom{a}{2}+\binom{b+c}{2}\right)\binom{a}{2}.
\]
Therefore without loss of generality, we can assume $b=0$, otherwise by changing the coloring, we could increase the product. 
Thus, we have obtained that $G_1$, $G_2$ and $G_3$ have the form of the construction in the proof of Theorem~\ref{main}. That is, we have a partition of the ground set into two parts $X$ and $Y$, and each edge inside $X$ is colored with colors $1$ and $2$, each edge inside $Y$ is colored with colors $2$ and $3$, and all edges between $X$ and $Y$ are colored with~$3$. The maximum product of the number of edges among such constructions is again given asymptotically by maximizing the expression~\eqref{opto}, thus yielding an upper bound of the form $\gamma n^6 \big( 1+o(1) \big)$.    
\end{proof}

Next we obtain some results about graphs avoiding a rainbow $P_4$ (recall $P_4$ denotes the path with $4$ vertices).

\begin{theorem} \label{thm:example_for_3_colors_and_no_rainbow_P4}
There exist graphs $G_1,G_2,G_3$ on a common vertex set of size $n$ and with no rainbow $P_4$ such that 
\[
e(G_1)e(G_2)e(G_3)\ge n^6 \left(\frac{1}{256}-o(1)\right). 
\]
\end{theorem}
\begin{proof}
 Let $[n]=X\cup Y$ be a partition of $[n]$ with $X$ and $Y$ of sizes approximately $0.5n$. Let $G_1$ consist of a complete graph on $X$, let $G_2$ consist of a complete graph on $Y$, and let $G_3$ consist of a complete graph on $X$ and a complete graph on $Y$.
 Observe that $e(G_1)e(G_2)e(G_3) = n^6 \big( \frac{1}{256}-o(1) \big)$.
 \end{proof}

\begin{theorem}
For three graphs $G_1,G_2,G_3$ on a common vertex set of size $n$ with no rainbow $P_4$, we have 
\[
e(G_1)e(G_2)e(G_3)\le n^6 \left(\frac{1}{256}+o(1)\right). 
\]
\end{theorem}
\begin{proof}
 
 Let $V$ be the vertex set of size $n$. Let $H$ be the graph induced by all edges with three colors. Then $H$ is $P_4$-free, so $e(H)\le \ex(n,P_4)\le n$. Hence we can assume that there is no edge with three colors. 

For any vertex $v\in V$ and for any $i \in [3]$, let $d_i(v)$ denote the degree of $v$ in $G_i$. For any $v \in V$, if $d_i(v)=o(n)$ for some $i\in [3]$, then we assume $d_i(v)=0$. Let $c(v)$ be the subset of  $\{1,2,3\}$ such that $d_i(v)\ge 1$ for every $i\in c(v) $ and $d_i(v)= 0$ for every $i\notin c(v)$.  
Let 
\[ V = \left( \bigcup_{i=1}^{3} A_i \right) \cup \left( \bigcup_{1\le i<j\le 3} A_{ij} \right) \cup A_{123} \]
be a partition of $V$, where $A_i$ is the set of vertices $v$ with $c(v)=\{i\}$, $A_{ij}$ is the set of vertices $v$ with $c(v)=\{i,j\}$ and  $A_{123}$ is the set of vertices $v$ with $c(v)=\{1,2,3\}$. We call the sets $A_i$, $A_{ij}$ and $A_{123}$ the parts of $G$.
Let $a_i=|A_i|$ for any $i\in [3]$, $a_{ij}=|A_{ij}|$ for any $1\le i<j\le 3$ and $a_{123}=|A_{123}|$.
If one part of this partition has size $o(n)$, then we assume that the size of this part is $0$. Note that there are no edges in $G[A_{123}]$, there are no edges between any two sets from $\{A_{12},A_{13},A_{23},A_{123}\}$, and there are no edges between $A_{i_1}$ and $A_{i_2}\cup A_{i_3}\cup A_{i_2i_3}$ for any $\{i_1,i_2,i_3\}=\{1,2,3\}$. 

Let $G_1,G_2,G_3$ be graphs on $V$ satisfying the assumptions above and maximizing $e(G_1)e(G_2)e(G_3)$. Let $G$ be the graph $G_1\cup G_2\cup G_3$. Then we may assume that $G[A_i]$ is complete in $G_i$ for any $i\in \{1,2,3\}$ and $G[A_{ij}]$ is complete in $G_i$ and $G_j$ for any $1\le i<j\le 3$. We may also assume that every vertex in $A_i$ is connected to every vertex in $A_{123}$ with an edge of color $i$ for $i\in [3]$ and every vertex in $A_{i,j}$ is connected to every vertex in $A_i$ and in $A_j$ with an edge of color $i$ and of color $j$, respectively.

Let $e_i=e(G_i)$ for any $i\in [3]$ and let 
\[ k(v):=\frac{d_1(v)}{e_1}+\frac{d_2(v)}{e_2}+\frac{d_3(v)}{e_3} \] for each $v\in V$. Then 
\[
\sum\limits_{v\in V}k(v)=\sum\limits_{v\in V}\sum\limits_{i=1}^3\frac{d_i(v)}{e_i}=\sum\limits_{i=1}^3\sum\limits_{v\in V}\frac{d_i(v)}{e_i}=6.
\]

\begin{lemma}\label{1}
	For any $v\in V$, we have
	\[ k(v)=\frac{6}{n} + o \left( \frac{1}{n} \right). \]
\end{lemma}
\begin{proof}
Let $u,v$ be two vertices in $V$. By maximality of $e_1 e_2 e_3$, putting $u$ from its part to the part of $v$, we obtain
\[
e_1e_2e_3\ge \big( e_1-d_1(u)+d_1(v)-1 \big) \big( e_2-d_2(u)+d_2(v)-1 \big) \big( e_3-d_3(u)+d_3(v)-1 \big).
\]
Hence $k(u)+ o \big( \frac{1}{n} \big) \ge k(v)$. Similarly, by putting $v$ from its part to the part of $u$, we obtain $k(v)+ o \big( \frac{1}{n} \big) \ge k(u)$. Hence $k(v)+ o \big( \frac{1}{n} \big) = k(u)$. Thus $k(v)=\frac{6}{n}+o \big( \frac{1}{n} \big)$ holds for any $v\in V$.
\end{proof}

By Theorem~\ref{thm:example_for_3_colors_and_no_rainbow_P4}, we may assume that $e_1e_2e_3\ge n^6/256$. Then $e_1+e_2+e_3\ge 3\sqrt[3]{n^6/256}>0.47n^2$. 

{\bf Case 1:}  $A_i$ is not empty for any $i\in [3]$.

Let $v_i$ be a vertex in $A_i$ for any $i \in [3]$, and let $\{i, j,k\}=\{1,2,3\}$. Applying Lemma~\ref{1} to $v_i$ for any $i\in [3]$, we obtain
\[
k(v_i)=\frac{a_i+a_{ij}+a_{ik}+a_{ijk}}{e_i}=\frac{6}{n}+o \left( \frac{1}{n} \right).
\]
Hence
\[ e_1+e_2+e_3=\frac{n}{6} \left( \sum\limits_{i=1}^{3}a_i+2\sum\limits_{1\le i<j\le 3}a_{ij}+3a_{123} \right)+o(n^2)>0.47n^2. \]
Then $a_{123}>0.8n$. Since there is no edge in $G[A_{123}]$, each edge is colored at most two times, thus $e_1+e_2+e_3\le 2 \big( \binom{n}{2}-\binom{0.8n}{2} \big)<0.47n^2$, a contradiction. 

{\bf Case 2:}  One of the parts $A_i$, say $A_1$, is empty and the other two are not.

First, observe that $A_{123}$ is empty. Applying Lemma \ref{1} to some vertex in $A_2$ and in $A_3$, we obtain
\[
\frac{a_2+a_{12}+a_{23}}{e_2}+o \left(\frac{1}{n} \right) =\frac{a_3+a_{13}+a_{23}}{e_3}+o \left( \frac{1}{n} \right) =\frac{6}{n}+o \left( \frac{1}{n} \right).
\]

Hence $e_2+e_3=\frac{n}{6}(n+a_{23})+o(n^2)$. Then $e_2e_3\le \big( \frac{n}{12} \big)^2(n+a_{23})^2+o(n^4)$. By $e_1\le (n-a_{23})^2/2$, we know 
\begin{align*}
e_1e_2e_3\le \left( \frac{n}{12} \right)^2 \frac{(n+a_{23})^2 (n-a_{23})^2}{2}+o(n^6)&\le n^6 \left( \frac{1}{288}+o(1) \right) \\
&< n^6 \left( \frac{1}{256}+o(1) \right).
\end{align*}

{\bf Case 3:}  Two of the parts $A_i$, say $A_1$ and $A_2$, are empty and the other one is not.

First, observe that $A_{123}$ is empty. If $A_{13}$ is empty, then moving all vertices in $A_3$ to $A_{23}$ would increase the product. Hence we may assume that $A_{13}$ and $A_{23}$ are not empty. Applying Lemma~\ref{1} to some vertex in $A_{13}$, in $A_{3}$ and in $A_{23}$, we obtain
\begin{align*}
\frac{a_3+a_{13}}{e_3}+\frac{a_{13}}{e_1}+o \left( \frac{1}{n} \right) =\frac{a_3+a_{13}+a_{23}}{e_3}+o \left( \frac{1}{n} \right) &=\frac{a_3+a_{23}}{e_3}+\frac{a_{23}}{e_2}+o \left( \frac{1}{n} \right) \\
&=\frac{6}{n}+o \left( \frac{1}{n} \right).
\end{align*}

Hence $\frac{a_{13}}{e_1} = \frac{a_{23}}{e_3} + o \big( \frac{1}{n} \big)$ and $\frac{a_{13}}{e_3} = \frac{a_{23}}{e_2} + o \big( \frac{1}{n} \big)$. So $e_3^{2}=e_1e_2+o(n^4)$ and $\frac{a_{13}e_2}{a_{23}}=\frac{a_{23}e_1}{a_{13}}+o(n^2)$. By $e_1=(a_{12}^2+a_{13}^2)/2+o(n^2)$ and $e_2=(a_{12}^2+a_{23}^2)/2+o(n^2)$, we have 
\[
a_{13}^2 \big( a_{12}^2+a_{23}^2 \big) = a_{23}^2 \big( a_{12}^2+a_{13}^2 \big) +o(n^4).
\]
Hence $a_{12}=0$ or $a_{13}=a_{23}+o(n)$. 
If $a_{12}=0$, then since $|A_3| > o(n)$, we know $e_3^{2}>e_1e_2+o(n^4)$, a contradiction. So we have $a_{13}=a_{23}+o(n)$.
Thus by 
\[ \frac{a_{13}}{e_1}=\frac{a_{23}}{e_3}+o \left( \frac{1}{n} \right),\]
we obtain $e_1=e_3+o(n^2)$.
Hence, 
\[
\frac{(a_{12}^2+a_{13}^2)}{2}=\frac{(n-a_{12})^2}{2}-a_{13}a_{23}+o(n^2)=\frac{(n-a_{12})^2}{2}-a_{13}^2+o(n^2).
\]

Then $(n^2-2na_{12})/2=3a_{13}^2/2+o(n^2)\le n^2 \big( 3/8+o(1) \big)$, so $n \big( 1/2 + o(1) \big) \ge a_{12}\ge n \big( 1/8-o(1) \big)$. It follows that
 \begin{align*}
	e_1+e_2+e_3 & = \frac{a_{12}^2 + a_{13}^2}{2} + \frac{a_{12}^2 + a_{23}^2}{2} + \left( \frac{(n-a_{12})^2}{2} - a_{13} a_{23} \right) + o(n^2) \\
	& = a_{12}^2 + \frac{(n-a_{12})^2}{2} + \frac{(a_{13} - a_{23})^2}{2} + o(n^2) \\
	& = a_{12}^2 + \frac{(n-a_{12})^2}{2} + o(n^2) \\	
	& = \frac{3}{2} \left( a_{12}-\frac{n}{3} \right)^2+\frac{n^2}{3}+o(n^2)\\
	& \le \frac{51n^2}{128}+o(n^2)<0.47n^2.
\end{align*}

{\bf Case 4:}  $A_i$ is  empty for all $i\in [3]$.

First, observe that $A_{123}$ is empty. Without loss of generality, we can assume $a_{12}\le a_{13}\le a_{23}$. Then the following holds:
\begin{multline*}
 \big( a_{12}^2+a_{13}^2 \big) \big( a_{13}^2+a_{23}^2 \big) \big( a_{12}^2+a_{23}^2 \big) \\
 \le \left( \left(a_{13}+\frac{a_{12}}{2} \right)^2+ \left( a_{23}+\frac{a_{12}}{2}  \right)^2 \right) \left( a_{13}+\frac{a_{12}}{2} \right)^2 \left( a_{23}+\frac{a_{12}}{2} \right)^2.
\end{multline*}
The preceding inequality is easy to verify by taking a difference of the right-~and left-hand sides, multiplying out and pairing off negative terms with positive ones, using $a_{12}\le a_{13}\le a_{23}$.
It follows that $a_{12}=0$ since otherwise we would increase the product $e_1e_2e_3$ by moving half of the vertices from $A_{12}$ to $A_{13}$ and the other half to $A_{23}$. 
Then 
\begin{align*}
e_1e_2e_3 = \frac{\big( a_{13}^2+a_{23}^2 \big) a_{13}^2a_{23}^2}{8} + o(n^6) &= \frac{\big( n^2-2a_{13}a_{23} \big) (a_{13}a_{23})^2}{8} + o(n^6) \\
&\le n^6 \left( \frac{1}{256}+o(1) \right).\qedhere
\end{align*}

\end{proof}

\begin{theorem} \label{thm:example_for_4_colors_and_no_rainbow_P4}
There exist graphs $G_1,G_2,G_3,G_4$ on a common vertex set of size $n$ and with no rainbow $P_4$ such that 
\[
e(G_1)e(G_2)e(G_3)e(G_4)\ge n^8 \left( \frac{1}{4096}-o(1) \right).
\]
\end{theorem}
\begin{proof}
Let $[n]=X\cup Y$ be a partition of $[n]$ with $X$ and $Y$ of sizes approximately $0.5n$. Let $G_1,G_2$ consist of a complete graph on $X$ and $G_3,G_4$ consist of a complete graph on $Y$. Observe that $e(G_1)e(G_2)e(G_3)e(G_4) = n^8 \big( \frac{1}{8^4}-o(1) \big)$.
\end{proof}

\begin{theorem}
For four graphs $G_1,G_2,G_3,G_4$ on a common vertex set of size $n$ and with no rainbow $P_4$, we have 
\[
e(G_1)e(G_2)e(G_3)e(G_4)\le n^8 \left( \frac{1}{8^4}+o(1) \right). 
\]
\end{theorem}

\begin{proof}
 Let $V$ be the vertex set. Let $H$ be the graph induced by all edges with at least three colors. 
 Then $H$ is $P_4$-free. Hence $e(H) \le \ex(n,P_4)\le n$. So we can assume that there is no edge with at least three colors. 

For any vertex $v\in V$ and for any $i \in [4]$, let $d_i(v)$ denote the degree of $v$ in $G_i$. For any $v \in V$, if $d_i(v)=o(n)$ for some $i\in [4]$, then we assume $d_i(v)=0$. Let $c(v)$ be the subset of  $\{1,2,3,4\}$ such that $d_i(v)\ge 1$ for every $i\in c(v) $ and $d_i(v)= 0$ for every $i\notin c(v) $.  
Let 
\[ V = \left( \bigcup_{i=1}^{4} A_i \right) \cup \left( \bigcup_{1\le i<j\le 4} A_{ij} \right) \cup A \]
be a partition of $V$, where $A_i$ is the set of vertices $v$ with $c(v)=\{i\}$, and $A_{ij}$ is the set of vertices $v$ with $c(v)=\{i,j\}$ and $A$ is the set of vertices $v$ with $|c(v)|\ge 3$. Let $a_i=|A_i|$ for any $i\in [4]$, $a_{ij}=|A_{ij}|$ for any $1\le i<j\le 4$ and $a=|A|$. If one part of this partition has size $o(n)$, then we assume the size of this part is~$0$. Note that there are no edges in $G[A]$, there are no edges between $A$ and $\bigcup_{1\le i<j\le 4}A_{ij}$, no edges between any two sets in $\{A_{ij}:1\le i<j\le 4\}$ and no edges between $A_i$ and $\big( \bigcup_{j\in [4]\setminus \{i\}} A_j \big) \cup \big( \bigcup_{ k,l \in [4]\setminus \{i\}: \, k<l} A_{kl})$.

Let $G_1,G_2,G_3,G_4$ be graphs on $V$  satisfying the assumptions above and maximizing the product of the size of the edge sets, and let $G$ be the graph $G_1\cup G_2\cup G_3\cup G_4$. 
By Theorem~\ref{thm:example_for_4_colors_and_no_rainbow_P4}, we may assume $e(G_1)e(G_2)e(G_3)e(G_4)\ge n^8 \big( 1/4096+o(1) \big)$. Therefore $e(G_1)+e(G_2)+e(G_3)+e(G_4) \ge n^2/2 + o(n^2)$. Since there are at most $n^2/2$ edges in $G$ and no edges with at least three colors, we know that the number of edges with two colors is at least as great as the number of missing edges in $G$.
Then 
\[
\sum\limits_{1\le i<j\le 4}\frac{a_{ij}^2}{2}>\sum\limits_{1\le i<j\le 4}\frac{a_{ij}}{2}(n-a_i-a_j-a_{ij}).
\]
Hence there are $i,j \in [4]$, say $1$ and $2$, such that $\frac{a_{12}^2}{2}>\frac{a_{12}}{2}(n-a_1-a_2-a_{12})$, so $2a_{12}+a_1+a_2>n$.
Let $B=A_1\cup A_2\cup A_{12}$ and $C=A_3\cup A_4\cup A_{34}$, let $b=|B|$ and let $c=|C|$. Then $b>n/2$ and $b+c\le n$. Clearly, there are no edges of color $3$ or $4$ adjacent to any vertex in $B$ and all edges of colors both $3$ and $4$ are in $G[A_{34}]$. 
It follows that
\[
e(G_3)+e(G_4)\le \frac{(n-a_{12}-a_1-a_2)^2}{2}+\frac{a_{34}^2}{2}\le \frac{(n-b)^2 + c^2}{2}.
\]

Clearly, there are no edges of color $1$ or $2$ adjacent to any vertex in $C$, all edges of colors both $1$ and $2$ are in $G[A_{12}]$, and there are no edges between $A_{12}$ and $V\setminus (B\cup C)$. 
Thus,
\[
e(G_1)+e(G_2)\le \frac{(n-a_{34}-a_3-a_4)^2}{2}+\frac{a_{12}^2}{2}-a_{12}(n-b-c).
\]
By $2a_{12}+a_1+a_2>n$, we have
\[
e(G_1)+e(G_2)\le \frac{(n-c)^2}{2}+\frac{b^2}{2}-b(n-b-c)=\frac{(n-b-c)^2 + 2b^2}{2}.
\]
Then
\[
e(G_1)e(G_2)e(G_3)e(G_4)\le \left(\frac{(n-b-c)^2+2b^2}{4}\right)^2 \left(\frac{(n-b)^2+c^2}{4}\right)^2.
\]
Since $(n-b)^2+c^2+(n-b-c)^2+2b^2\le 2(n-b)^2+2b^2$ and $(n-b-c)^2+2b^2\ge 2b^2\ge (n-b)^2+b^2$, we obtain
\[
((n-b)^2+c^2)((n-b-c)^2+2b^2)\le 4(n-b)^2b^2\le \frac{n^4}{4}.
\]
Thus $e(G_1)e(G_2)e(G_3)e(G_4)\le \frac{n^8}{8^4}$.
\end{proof}



We conclude by mentioning some open problems.  
First, it appears plausible that for $6$ graphs $G_1,G_2,\ldots,G_6$ avoiding a rainbow $K_4$, the asymptotically optimal construction is simply $6$ copies of the Tur\'an graph with $3$ parts.  

Second, for $2k$ graphs avoiding a rainbow $P_4$, one can take an independent set of size $n(2k-1)/(4k-1)$ and $2k$ cliques, each of size $n/(4k-1)$, each containing edges of one of the colors, as well as the edges between each of the cliques and the independent set in their respective colors.  It appears this construction may be asymptotically optimal.

Finally, for $k$ graphs avoiding a rainbow path $P_{k+1}$ ($k+1$ vertices), one could take two disjoint cliques of size $n/2$, each consisting of $k$-colored edges from two different sets of $k$ colors.   

\section*{Acknowledgments}
The research of Frankl, Gy\H{o}ri and Salia was partially supported by the National Research, Development and Innovation Office NKFIH, grants  K132696, SNN-135643 and K126853. 
The research of Salia was supported by the Institute for Basic Science (IBS-R029-C4). 
The research of Tompkins was supported by NKFIH grant K135800.
\vspace{-.7em}

\end{document}